\RequirePackage{fix-cm}
\documentclass[smallextended]{svjour3}
\smartqed

\usepackage{amsfonts}
\usepackage{amssymb}
\usepackage{amsmath}
\usepackage{amsopn}
\usepackage{graphicx}
\usepackage{epstopdf}
\usepackage{algorithm,algorithmicx}
\usepackage{algpseudocode}
\usepackage{comment}
\usepackage{color}
\usepackage{bm}
\usepackage{mathtools}
\usepackage{bbm}
\usepackage{booktabs}
\usepackage[numbers,sort&compress]{natbib}
\usepackage{lineno}
\usepackage{url}
\usepackage[hidelinks]{hyperref}

\DeclareGraphicsExtensions{.eps,.pdf,.png,.jpg}

\newcommand{\bphi}{\bm{\phi}}
\newcommand{\bu}{\bm{u}}

\newcommand{\bb}{\bm{b}}

\newcommand{\m}{\bm{m}}

\newcommand{\bv}{\bm{v}}

\newcommand{\h}{\bm{h}}
\newcommand{\md}{\,\mathrm{d}}
\newcommand\bx{\bm{x}}
\newcommand\by{\bm{y}}
\newcommand\bz{\bm{z}}
\newcommand\bw{\bm{w}}
\newcommand\bp{\bm{p}}
\newcommand\bs{\bm{s}}
\newcommand{\pt}{\partial_t}

\newcommand{\mathup}[1]{\mathrm{#1}}
\DeclarePairedDelimiter{\abs}{\lvert}{\rvert}
\providecommand{\coloneq}{\coloneqq}

\DeclareMathOperator{\diag}{diag}
\journalname{Journal of Scientific Computing}
\begin{document}
\title{Matrix-Free FFT--HSS Preconditioning for Periodic Landau--Lifshitz--Gilbert Saddle-Point Systems}
\titlerunning{FFT--HSS Preconditioning for Periodic LLG Systems}

\author{Hang Qi \and Changqing Ye \and Xiaofei Guan}
\authorrunning{H. Qi et al.}

\institute{
H. Qi \at
School of Mathematical Sciences, Key Laboratory of Intelligent Computing and Applications
(Ministry of Education), Tongji University, Shanghai 200092, China\\
\email{qihang7750@outlook.com}
\and
C. Ye \at
Department of Mathematics, The Chinese University of Hong Kong, Shatin,
Hong Kong Special Administrative Region, China\\
\email{ye\_changqing@outlook.com}
\and
X. Guan \at
School of Mathematical Sciences, Key Laboratory of Intelligent Computing and Applications
(Ministry of Education), Tongji University, Shanghai 200092, China\\
\email{guanxf@tongji.edu.cn}
}

\date{}
\maketitle

\begin{abstract}
In this paper, a matrix-free Hermitian/skew-Hermitian splitting (HSS)
preconditioner is proposed for periodic Landau--Lifshitz--Gilbert (LLG)
saddle-point systems.
The main contributions are threefold. (1) The coupled skew/constraint block
has an explicit \(4\times4\) nodal inverse and requires no local factorization.
Combining this local inverse with FFT inversion of the shifted exchange block
gives \(O(N_g\log N_g)\) work and \(O(N_g)\) temporary storage per application.
(2) We establish well-posedness and an even-step GMRES residual bound, with an
iteration estimate uniform in the mesh size and time step for a class of
coupled refinements. (3) The projected implicit Euler and projection-free
Crank--Nicolson-type midpoint discretizations generate saddle-point systems of
the same form, so the same matrix-free FFT--HSS preconditioning procedure
applies to both. They achieve first- and second-order temporal accuracy,
respectively; for quadratic-affine energies, the midpoint scheme also
preserves nodal length and satisfies an exact discrete dissipation identity.
Two- and three-dimensional experiments confirm the predicted temporal orders
and midpoint invariants. On the largest smooth tests, FFT--HSS reduces GMRES
iterations by \(35\)--\(39\%\) relative to same-grid Householder
preconditioning and yields approximately 15-fold and 3-fold speedups over
unpreconditioned GMRES in two and three dimensions, respectively. Broadband
tests retain mesh-independent iterations in the covered refinement regime at
time steps 13.3 times the linearized explicit exchange limit.
\end{abstract}


\section{Introduction}\label{sec1}
The Landau--Lifshitz--Gilbert (LLG) equation models magnetization dynamics
\cite{LL:E:1935,R:VTV:2001,FA:NATMA:1992,C:ACME:2007,KP:SR:2006}. Numerical
schemes must handle its unit-length constraint, dissipation, nonlinearity, and
nonlocal lower-order fields. Tangent-plane schemes solve for a velocity in the
current tangent space
\cite{A:DCDS-S:2008,AJ:MMMAS:2006,AKT:PBCM:2012,GR:SJNA:2017,YCH:JCP:2021};
their convergence and coupled extensions are studied in
\cite{BPPR:IJNA:2014,AHPPRS:CMA:2014,FT:SJNA:2017}. Implicit midpoint variants
are second order and preserve nodal length without projection
\cite{BP:SJNA:2006,AKST:NM:2014,A:JSC:2016,AGS:SJNA:2021}. The stiff exchange
Laplacian imposes an \(h^2\)-scale restriction on explicit and
implicit--explicit treatments \cite{BKP:MC:2007,FHV:ANM:1997,PRS:CMA:2018},
whereas implicit exchange permits larger steps. Its benefit depends on
efficient solution of the resulting constrained systems.

Linear algebra is therefore a principal cost rather than an incidental
implementation detail. Every tangent-plane step is a nonsymmetric constrained
solve, repeated within midpoint nonlinear iterations
\cite{Elman:YU:1982}. Kraus et al.\
\cite{KrausEtAl2019} addressed this issue
using Householder reflections to construct effective preconditioners for
finite-element tangent-plane systems and proved linear convergence of
preconditioned GMRES. Related block preconditioners and finite-element
micromagnetic codes have also been developed
\cite{FCVKML:ITM:2019,AEBDS:JMMM:2013}. These works establish the value of
structure-aware solvers. For the periodic setting
considered here, they raise a distinct question: can Fourier diagonalization and the nodal
tangent constraint be treated together without assembling a global matrix or
factorizing local blocks, while retaining convergence information for the full
saddle-point system?

Periodic boundary conditions are not introduced solely for FFT
diagonalization. They arise in representative-cell simulations of bulk or
repeating magnetic media, including extended wires and films, periodic
microstructures, and spin-wave propagation, and they can eliminate artificial
outer-boundary effects that would otherwise obscure the response of the
material microstructure
\cite{BDHAS:SR:2021,ADL:JMMM:2025}. Within this application class, we consider
a uniform Cartesian grid. Periodicity diagonalizes the exchange operator by
the discrete Fourier transform \cite{GVZPG:CMAME:2017,YFC:CMAME:2024}, while
the nodal tangent constraint has a
small, explicitly invertible local structure. We use these two facts to
specialize the classical Hermitian/skew-Hermitian splitting (HSS)
\cite{BG:SJMAA:2004} to the generalized saddle-point system. The shifted
Hermitian factor is applied by FFT-based Helmholtz solves, and the coupled
skew/constraint factor has an explicit \(4\times4\) nodal inverse, so no local
factorization is required. The resulting split FFT--HSS preconditioner is
fully matrix-free and requires \(O(N_g\log N_g)\) work and \(O(N_g)\)
temporary storage per application, where \(N_g\) is the number of grid points.
A symmetric-part FFT preconditioner is retained as a comparison that isolates
the effect of the skew/constraint factor. The central contribution is this
periodic algebraic construction rather than a general-purpose micromagnetic
solver.

The analysis connects this construction to both the linear solver and the
time discretization. We establish well-posedness of the discrete saddle
system, recall the classical contraction result for the ideal
tangent-reduced HSS iteration, and prove an even-step residual estimate for
unrestarted GMRES on the full left-preconditioned saddle system. For a class of
coupled mesh--time-step refinements with uniform nodal bounds, the estimate
yields an explicit residual-tolerance iteration bound independent of the mesh
and time step. This condition governs solver complexity rather than stability;
the tested regime includes time steps beyond the linearized explicit exchange
limit. A complementary constant-state analysis characterizes the scope of the
uniform bound for the stationary splitting. An inexact time-stepping estimate
separates temporal, spatial, and algebraic errors.

The same matrix-free solver is used for a projected implicit Euler scheme and
a projection-free Crank--Nicolson-type midpoint scheme, referred to below as
the midpoint scheme; the two discretizations therefore require no separate
linear-algebra pipelines. For time-dependent or nonlinear lower-order fields,
midpoint evaluation is performed inside the Picard iteration, and a local
contraction result supplies a sufficient time-step condition for that
iteration. For quadratic-affine energies, the converged midpoint step
preserves nodal length and satisfies an exact discrete dissipation identity.

Two- and three-dimensional experiments are organized to test the accuracy,
efficiency, and parameter dependence of the complete procedure. They examine
temporal convergence, preconditioner performance for smooth and random
right-hand sides, dependence on the shift parameter, and discrete energy and
length conservation. The comparison includes an implementation of the
Householder tangent-space preconditioner of \cite{KrausEtAl2019} on the same
Fourier grid. For smooth data, split FFT--HSS gives lower Krylov iteration
counts than this reference and substantial solve-time gains over
unpreconditioned GMRES on the largest grids. The broadband experiment
contrasts fixed-step and bounded-stiffness refinement, thereby testing both
the prediction and the limitation of the uniform iteration analysis.

The remainder of the paper is organized as follows. Section~\ref{sec2}
introduces the periodic LLG model and its saddle-point discretization.
Section~\ref{sec3} derives the split FFT--HSS preconditioner and the midpoint
scheme. Section~\ref{sec4} presents the numerical experiments, and
Section~\ref{sec5} gives the conclusions and limitations.

\section{Preliminaries}\label{sec2}
This section fixes the periodic LLG model, the periodic magnetostatic operator,
and the Fourier collocation space used throughout the analysis. We then derive
the projected tangent-plane Euler step and its generalized saddle-point form.

\subsection{Model problem}\label{sec2.1}
We consider a periodic rectangular cell
\(\Omega=\prod_{\ell=1}^d(0,L_\ell)\), \(d\in\{2,3\}\), identified with
the flat torus. The magnetization field \(\bm{M}(\bx,t)\) is periodic in
each spatial coordinate.
Without loss of generality, we consider the normalized magnetization field \(\m(\bx,t) = \bm{M}(\bx,t)/M_\mathup{s}\), where \(M_\mathup{s}\) is the saturation magnetization.
The dynamics of the magnetization \(\m(\bx,t)\) can be described by the following normalized LLG equation
\begin{equation}\label{eq:main LLG}
    \pt\m = -\m \times \h_{\mathup{eff}} - \alpha \m \times \left(\m\times \h_{\mathup{eff}}\right)
\end{equation}
where the damping constant satisfies \(\alpha>0\), and the normalized effective
field \(\h_{\mathup{eff}}\) is given by the variational derivative of the
normalized micromagnetic free energy \(g(\m)\). We prescribe periodic initial
data \(\m(\cdot,0)=\m_0\) with \(|\m_0|=1\).
The LLG equation \eqref{eq:main LLG} can be equivalently written in another form
\begin{equation}\label{tangent plane form}
    \alpha\pt\m+\m\times\pt\m=(1+\alpha^2)\left[\h_{\mathrm{eff}}-(\m\cdot\h_{\mathrm{eff}})\m\right].
\end{equation}
Indeed, the LLG flow preserves \(|\m|=1\). Using this unit-length identity and
\(\bm{a} \times (\bm{b} \times \bm{c}) = (\bm{a} \cdot \bm{c})\bm{b}
- (\bm{a} \cdot \bm{b})\bm{c}\), we have
\(-\m \times (\m \times \h_{\mathrm{eff}} )
= \h_{\mathrm{eff}} - (\m \cdot \h_{\mathrm{eff}} )\m\).
Taking the vector product of \eqref{eq:main LLG} with \(\m\) and adding
\(\alpha\) times \eqref{eq:main LLG} then yields
\eqref{tangent plane form}.

The free energy \(g(\m)\) consists of normalized exchange, demagnetization, anisotropy, and Zeeman energy, written as
\begin{equation*}
    g(\m) = \int_{\Omega} \frac{C_{\mathup{exc}}}{\mu_0 M_\mathup{s}^2} \abs{\nabla\m}^2 - \frac{1}{2} \m \cdot \h_{\mathup{dem}}[\m] + \frac{K_u}{\mu_0 M_\mathup{s}^2} \big(1 - (\bm{e}_{\mathup{an}} \cdot \m)^2\big) - \m \cdot \h_{\mathup{ext}} \md V.
\end{equation*}
Here, \(C_{\mathup{exc}}\) represents the exchange stiffness constant, \(\mu_0\) denotes the magnetic permeability of vacuum, \(\h_{\mathup{dem}}[m]\) is the demagnetization field given by the magnetostatic Maxwell equations, \(K_u\) is the uniaxial anisotropy constant, \(\bm{e}_{\mathup{an}}\) is the unit vector of the easy axis, and \(\h_{\mathup{ext}}\) is the external magnetic field. 
Note that $\m\cdot \h_{\mathup{dem}}[\m]$ is a quadratic form.
Consequently, the effective field \(\h_{\mathup{eff}}\) can be represented as
\begin{align}\label{heff}
    \h_{\mathup{eff}}=-\frac{\delta g(\m)}{\delta \m}=\frac{2C_{\mathup{exc}}}{\mu_{0}M_{s}^{2}}\Delta \m + \h_{\mathup{dem}}[\m]+\frac{2K_{u}}{\mu_{0}M_{s}^{2}}\bm{e}_{\mathup{an}}(\m\cdot \bm{e}_{\mathup{an}})+\h_{\mathup{ext}}.
\end{align}
For simplicity, we introduce the following notations
\begin{align*}
    \begin{aligned}
        \h_{\mathup{low}}[\m] & \coloneqq  \h_{\mathup{dem}}[\m]+\h_{\mathup{ani}}[\m]+\h_{\mathup{ext}}\\
        & = \h_{\mathup{dem}}[\m] + \frac{2K_{u}}{\mu_{0}M_{s}^{2}}\bm{e}_{\mathup{an}}(\m\cdot \bm{e}_{\mathup{an}})+\h_{\mathup{ext}}.
    \end{aligned}
\end{align*}
We will detail the governing equation of the demagnetization field and the boundary conditions in Section \ref{sec2.4}.
The pointwise constraint \(\abs{\m}=1\) for all \((\bx,t)\) is a defining
property of the LLG equation and must be respected by the time discretization.

\subsection{The magnetostatic problem}\label{sec2.4}
In this subsection, we focus on the demagnetization field and the associated magnetostatic problem. As the external magnetic field and the easy-axis direction are usually given, the remaining contributions to \(\h_{\mathup{low}}\) are prescribed.

In magnetostatics, the magnetic field is treated as quasi-static and generated solely by the magnetization of the material, with external currents neglected. 
In this case, the demagnetization field \(\bm{H}_{\mathup{dem}}\) arises from the magnetization \(\bm{M}(\bx,t)\) and satisfies the magnetostatic Maxwell equations
\begin{equation*}
    \begin{cases}
        \operatorname{div} \bm{B} = 0, & \text{(Gauss's law)}, \\
        \operatorname{curl} \bm{H}_{\mathup{dem}} = 0, & \text{(Ampere's law)},
    \end{cases}
\end{equation*}
where \(\bm{B} = \mu_0(\bm{M}+\bm{H}_{\mathup{dem}})\) is the magnetic flux density and \(\mu_0\) here denotes the permeability of the vacuum. 

Since \(\operatorname{curl} \bm{H}_{\mathup{dem}} = 0\), there is a scalar potential \(\varphi(\bx)\) such that \(\bm{H}_{\mathup{dem}} = -\nabla\varphi\). Substituting \(B = \mu_0(\bm{M}-\nabla\varphi)\) into Gauss's law, we obtain the following partial differential equation
\begin{equation*}
    \Delta \varphi=\operatorname{div} \bm{M},\qquad \text{in } \mathbb{R}^3.
\end{equation*}
Similarly, for the normalized magnetization \(\m\) and the corresponding demagnetizing field \(\h_{\mathup{dem}}\), the governing equation is given by
\begin{equation}\label{hdem}
    \left\{\begin{aligned}
        &\h_{\mathup{dem}}(\bx) = -\nabla\phi(\bx), \\
        & \Delta \phi =\operatorname{div}\m.
    \end{aligned}
    \right.
\end{equation}
where \(\phi(\bx)\) denotes the scalar potential in the normalized variables.

For an isolated sample, \eqref{hdem} is instead posed in free space with
\(\phi = \mathcal{O}(1/\abs{\bm{r}})\) as \(\bm{r}\to\infty\), and it can
be represented by
\begin{equation*}
    \h_{\mathup{dem}} = -\nabla \int_{\Omega} \nabla N(\bx - \bm{y}) \cdot \m(\bm{y}) \, \mathrm{d}\bm{y},
\end{equation*}
where \(N(\bx) = -1/(4\pi|\bx|)\) denotes the Newtonian potential.

The Maxwell formulation above is used only for the three-dimensional
periodic demagnetizing field in the micromagnetic test problem. The two-dimensional
studies below use manufactured forcing and do not claim a thin-film
stray-field model. The present work does not approximate that free-space
problem. We impose
periodic boundary conditions on both \(\m\) and \(\phi\), together with the
zero-mean convention for \(\phi\) \cite{BDHAS:SR:2021,ADL:JMMM:2025}.
The periodic Poisson operator is diagonal in Fourier space. Open-boundary FFT
convolutions generally use zero padding and a different discrete kernel; they
are therefore outside the operator and cost model analyzed below. Periodic
cells are appropriate for bulk or representative-cell calculations, but not
for arbitrary finite devices.

\subsection{Periodic spectral discretization}\label{sec2.2}
Let \(\mathcal G_N\) be a uniform periodic Cartesian grid with \(N_\ell\)
points in direction \(\ell\), and let \(N_g=\prod_{\ell=1}^dN_\ell\).
We denote the corresponding trigonometric-polynomial space by \(V_N\) and
write
\[
 (u,v)_N=\frac{|\Omega|}{N_g}\sum_{\bx_j\in\mathcal G_N}
 u(\bx_j)\cdot v(\bx_j)
\]
for the discrete inner product. Fourier differentiation defines
\(\nabla_N\) and \(\Delta_N\). For a Fourier coefficient with wave vector
\(\bm{\xi}\), the positive operator \(-\Delta_N\) has symbol
\(|\bm{\xi}|^2\), with the physical scaling contained in
\(\bm{\xi}_\ell=2\pi k_\ell/L_\ell\).

The admissible nodal magnetizations and the discrete tangent space are
\[
 \mathcal M_N=\{\bphi_N\in V_N^3:\ |\bphi_N(\bx_j)|=1
 \ \text{for every }\bx_j\in\mathcal G_N\},
\]
\begin{equation}\label{mod-tangent}
 \mathcal K_N(\m_N)=
 \{\bphi_N\in V_N^3:\ \m_N(\bx_j)\cdot\bphi_N(\bx_j)=0
 \ \text{for every }\bx_j\in\mathcal G_N\}.
\end{equation}
Thus the tangent constraint is imposed pointwise at grid nodes. For
compatibility with the variational notation below, we set
\(H_h^1=V_N^3\), \(Q_h=V_N\), and use the subscript \(h\) for grid
functions when no confusion can arise. The essential computational point is
that the exchange block is diagonal in Fourier space, whereas multiplication
by \(\m_N\) is local in physical space.

\subsection{Projected implicit Euler scheme}\label{sec_euler}
We first discretize in time by a projected implicit Euler tangent-plane
scheme. Let
\(0=t_0<t_1<\dots<t_{N_t}=T\) be a uniform partition with
\(k=t_{i+1}-t_i\), and write
\(\bm g^i=\bm g(t_i)\) for any time-dependent grid function \(\bm g\).

Following the general approach of \cite{A:DCDS-S:2008,AJ:MMMAS:2006} with the modified discrete tangent space \eqref{mod-tangent}, we discretize \eqref{tangent plane form} by determining the approximated time derivative \(\bv^{n+1}_h\approx\pt\m^{n+1}_h \in \mathcal{K}_h(\m^n_h)\) at the \(n\)-th time step such that
\begin{equation}\label{total system}
    \begin{aligned}
        &\quad \underbrace{\alpha(\bv_{h}^{n+1},\bphi_{h})_{h}}_{\Pi_1}+\underbrace{(\m_{h}^{n}\times\bv_{h}^{n+1},\bphi_{h})_{h}}_{\Pi_2}+\underbrace{k\ \widehat{C}_{\mathup{exc}} (\nabla\bv_h^{n+1},\nabla\bphi_h)}_{\Pi_3}\\
        &=\underbrace{-\widehat{C}_{\mathup{exc}}(\nabla\m_{h}^{n},\nabla\bphi_{h})+(1+\alpha^2)(\h_{\mathup{low}}(\m^n_h),\bphi_h)_h}_{\Pi_4}, \quad \forall \bphi_h\in \mathcal{K}_h(\m_h^n)
    \end{aligned}  
\end{equation}
where \(\widehat{C}_{\mathup{exc}}\coloneqq 2C_{\mathup{exc}}(1+\alpha^2)/(\mu_{0}M_{s}^{2})\). 

\begin{remark}
The exchange Laplacian is the stiff component of
\eqref{tangent plane form} \cite{PRS:CMA:2018}. Equation~\eqref{total system}
therefore treats exchange implicitly and the lower-order field explicitly,
producing a linearly implicit first-order tangent-plane step. This choice
reduces each time step to one linear saddle solve; no unconditional stability
claim for the explicit lower-order treatment is used below.
\end{remark}

As both the unknown and the test functions in \eqref{total system} are constrained to the tangent space, it is natural to formulate the problem as a saddle-point system: Find \((\bv^{n+1}_h,r_h)\in H_h^1\times Q_h\), such that
\begin{equation}\label{total system 2}
    \left\{
    \begin{aligned}
        &\alpha(\bv_{h}^{n+1},\bphi_{h})_{h}
        +(\m_{h}^{n}\times\bv_{h}^{n+1},\bphi_{h})_{h}\\
        &\quad+k\widehat{C}_{\mathup{exc}}
        (\nabla\bv_h^{n+1},\nabla\bphi_h)
        +(\m_h^n\cdot\bphi_h,r_h)_h\\
        &\quad=-\widehat{C}_{\mathup{exc}}
        (\nabla\m_{h}^{n},\nabla\bphi_{h})
        +(1+\alpha^2)(\h_{\mathup{low}}(\m^n_h),\bphi_h)_h,\\
        &(\m_h^n\cdot\bv^{n+1}_h,q_h)_h=0. 
    \end{aligned}
    \right. 
\end{equation}
Equation~\eqref{total system 2} holds for all
\((\bm{\phi}_h,q_h)\in H_h^1\times Q_h\).
An equivalent matrix formulation is given as
\begin{equation}\label{saddle point system}
    \underbrace{\left[\begin{aligned}
        &A\quad B^T\\
        -&B\quad O
    \end{aligned}
    \right]}_{\mathcal{A}}\left[\begin{aligned}
        &\bv\\
        &\bm{r}
    \end{aligned}
    \right]
    =\underbrace{\left[\begin{aligned}
        &\bm{f}\\
        &\bm{0}
    \end{aligned}
    \right]}_{\bm{b}}.
\end{equation}
Here, the matrix \(A\) comprises the sum \(\Pi_1 + \Pi_2 + \Pi_3\) and the vector \(\bm{f}\) coincides with \(\Pi_4\), with each \(\Pi_i\) (\(i=1, \dots, 4\)) as defined in \eqref{total system}. 
Additionally, \(B\) is a block diagonal matrix
\begin{equation*}
    B=\begin{bmatrix}
      B_1 &         &        & \\
             & B_2  &        & \\
             &         & \ddots & \\
             &         &        & B_N 
  \end{bmatrix},
\end{equation*}
where the block \(B_i\) is a \(1\times 3\) matrix
\begin{equation}\label{Bi}
    B_i = [m_{i;1},m_{i;2},m_{i;3}],
\end{equation}
where \(m_{i;j}\) denotes the \(j\)-th component of the magnetization at grid
point \(i\). The constant quadrature weight has been divided out of both
block rows.

\begin{remark}
Directly from \eqref{total system 2}, the \((2,1)\)-block of the coefficient matrix \(\mathcal{A}\) for the resulting saddle-point system is \(B\).
However, to align the algebraic form with the splitting iteration framework in \cite{BG:SJMAA:2004}, we multiply the second block row of the system by
\(-1\). 
Consequently, the \((2,1)\)-block in \eqref{saddle point system} is written as \(-B\). 
Since the corresponding constraint equation is homogeneous, this row scaling does not change the feasible set or the solution of the saddle-point system.
\end{remark}

\begin{proposition}[Well-posedness of each linearized saddle system]
\label{prop:saddle-wellposed}
Let \(H=H^T\succ0\), \(S^T=-S\), and suppose every nodal
magnetization is nonzero. Then \(B\) in \eqref{Bi} has full row rank and
\begin{equation}\label{eq:B-singular-values}
 BB^T=\diag\bigl(\|\m_1\|_2^2,\ldots,\|\m_{N_g}\|_2^2\bigr),\qquad
 \sigma_{\min}(B)=\min_i\|\m_i\|_2 .
\end{equation}
In particular, the discrete constraint has inf-sup constant one for nodal
unit magnetizations. The matrix
\[
 \mathcal A=
 \begin{bmatrix}H+S&B^T\\-B&0\end{bmatrix}
\]
is nonsingular, as in the standard generalized saddle-point argument
\cite{BG:SJMAA:2004}. Moreover, for every \(\lambda>0\), both shifted factors
\(\mathcal H+\lambda\mathcal I\) and
\(\mathcal S+\lambda\mathcal I\) in \eqref{HS} are nonsingular.
\end{proposition}

\begin{proof}
The block identity gives \eqref{eq:B-singular-values}. Testing the homogeneous
first equation by \(\bv\), with \(B\bv=0\), gives
\(\bv^TH\bv=0\), hence \(\bv=0\) and then \(\bm r=0\).
The shifted-factor claim follows from positive definiteness of
\(\mathcal H+\lambda\mathcal I\) and
\(\|(\lambda\mathcal I+\mathcal S)\bx\|_2^2
=\lambda^2\|\bx\|_2^2+\|\mathcal S\bx\|_2^2\).
\end{proof}

\section{Numerical method and convergence analysis}\label{sec3}
The construction begins with the Hermitian and skew-Hermitian parts of the
generalized saddle matrix. The Fourier symbol of the Hermitian block and the
exact local inverse of the coupled skew/constraint block then define the
matrix-free preconditioner. The same linear solver is incorporated into the
projection-free midpoint scheme. The analysis treats
well-posedness of the saddle system, convergence of the tangent-reduced HSS
iteration, GMRES convergence for the full saddle-point system, and the effect
of inexact algebraic solves on the time discretizations.

\subsection{HSS iteration}\label{sec3.1}
The HSS iteration for the generalized saddle-point system
\(\mathcal{A}\bx=\bb\) was introduced in \cite{BG:SJMAA:2004}. Write
\(\mathcal{A}=\mathcal{H}+\mathcal{S}\), where
\[
\mathcal{H} = \frac{1}{2} (\mathcal{A} + \mathcal{A}^T), \quad \mathcal{S} = \frac{1}{2} (\mathcal{A} - \mathcal{A}^T)
\]
For \(\lambda>0\), the two equivalent splittings are
\[
\mathcal{A} = (\mathcal{H} + \lambda \mathcal{I}) - (\lambda \mathcal{I} - \mathcal{S}) \text{ and } \mathcal{A} = (\mathcal{S} + \lambda \mathcal{I}) - (\lambda \mathcal{I} - \mathcal{H})
\]
where \(\mathcal{I}\) is the identity matrix. Given an initial guess
\(\bx^{(0)}\), the HSS iteration is
\begin{align}\label{stationary iteration}
    \begin{cases} 
        (\mathcal{H} + \lambda \mathcal{I}) \bx^{(k + \frac{1}{2})} = (\lambda \mathcal{I} - \mathcal{S}) \bm{x}^{(k)} + \bb, \\
        (\mathcal{S} + \lambda \mathcal{I}) \bx^{(k + 1)} = (\lambda \mathcal{I} - \mathcal{H}) \bx^{(k + \frac{1}{2})} + \bb.
    \end{cases}
\end{align}

For the saddle-point system \eqref{saddle point system}, the two parts are
\begin{align}\label{HS}
    \mathcal{H} = \begin{bmatrix}
    H & O \\
    O & O
    \end{bmatrix} \quad \text{and} \quad \mathcal{S} = \begin{bmatrix}
    S & B^T \\
    -B & O
    \end{bmatrix},
\end{align}
where $H=\Pi_1+\Pi_3$, $S=\Pi_2$. The first half-step of \eqref{stationary iteration} relates to two (uncoupled) linear systems of the form
\begin{align}\label{first half-step}
    \begin{cases}
    (H + \lambda I_{3N}) \bv^{(k + \frac{1}{2})} = \lambda \bv^{(k)} - S\bv^{(k)} + \bm{f} - B^T \bm{r}^{(k)}, \\
    \lambda \bm{r}^{(k + \frac{1}{2})} = \lambda \bm{r}^{(k)} + B \bv^{(k)}.
    \end{cases}
\end{align}
The first equation in \eqref{first half-step} is a shifted periodic Helmholtz
problem and is diagonal in Fourier space.

The second half-step of \eqref{stationary iteration} is
\begin{equation}\label{sh-si}
    (\mathcal{S} + \lambda \mathcal{I}) \bx^{(k + 1)} = (\lambda \mathcal{I} - \mathcal{H}) \bx^{(k + \frac{1}{2})} + \bb.
\end{equation}
Introduce the reordered vector
\[
\tilde{\bm{x}}^{(k)}=[v_{1;1}^{(k)},v_{1;2}^{(k)},v_{1;3}^{(k)},r^{(k)}_1,\cdots,v_{N_g;1}^{(k)},v_{N_g;2}^{(k)},v_{N_g;3}^{(k)},r^{(k)}_{N_g}]^T
\]
which groups the three velocity components and the Lagrange multiplier at
each grid point. Under this ordering, \eqref{sh-si} becomes
\[
(\tilde{\mathcal{S}} + \lambda \mathcal{I}) \tilde{\bx}^{(k + 1)} = (\lambda \mathcal{I} - \tilde{\mathcal{H}}) \tilde{\bx}^{(k + \frac{1}{2})} + \tilde{\bb},
\]
where \(\tilde S\) is block diagonal with nodal blocks
\begin{equation}\label{S}
    \tilde{S}_i=
    \left(
    \begin{aligned}
        &\Lambda_i &B_i^T\\
        -&B_i & O
    \end{aligned}
    \right)
\end{equation}
where \(\Lambda_i\) is defined as
\begin{equation}\label{Lambda}
    \Lambda_i=\begin{bmatrix}
        0 & -m_{i;3} & m_{i;2} \\
        m_{i;3} & 0 & -m_{i;1} \\
        -m_{i;2} & m_{i;1} & 0
    \end{bmatrix}
\end{equation}
Thus applying \((\mathcal S+\lambda\mathcal I)^{-1}\) requires one
independent \(4\times4\) solve per grid point. The inverse is available in
closed form, so no local factorization is required.

\subsection{FFT--HSS preconditioning of GMRES}\label{sec3.2}
We use the two shifted factors to precondition GMRES rather than use
\eqref{stationary iteration} as the outer solver. This distinction matters:
the reduced stationary optimum derived below need not minimize iterations or
time for the full saddle system.
We use
\(\mathcal{M}_\lambda
=\mathcal{M}_{\lambda}^{(1)}\mathcal{M}_{\lambda}^{(2)}
\coloneq(\mathcal{H}+\lambda\mathcal{I})
(\mathcal{S}+\lambda\mathcal{I})\)
as a left preconditioner. Its application requires the successive solves
\((\mathcal{H}+\lambda\mathcal{I})\bm w=\bm p\) and
\((\mathcal{S}+\lambda\mathcal{I})\bz=\bm w\).

The first solve uses the inverse of
\(\mathcal{M}_{\lambda}^{(1)}
=\mathcal{H}+\lambda\mathcal I\).
Recalling \eqref{total system} and \eqref{HS}, the Hermitian block contains
only the identity and the periodic Laplacian. Parseval's identity gives
\[
 (-\Delta_N\bv_h,\bphi_h)_N
 =\sum_{\bm{\xi}\in\mathcal F_N}
 |\bm{\xi}|^2\widehat{\bv}_h[\bm{\xi}]
 \cdot\overline{\widehat{\bphi}_h[\bm{\xi}]} .
\]
Thus the Laplacian is diagonal in Fourier space. After division by the
constant grid quadrature weight, the zeroth-order block is the identity, so
\(\mathcal{H}+\lambda\mathcal I\) is inverted mode by mode.
Algorithm~\ref{fh_pre} summarizes its matrix-free application.

\begin{algorithm}[!ht]
    \caption{Matrix-free solution of
    \(\mathcal{M}_{\lambda}^{(1)}\bm{w}=\bm{p}\).}
    \label{fh_pre}
    \begin{algorithmic}[1]
        \Require \(\bm{p}=(\bm{v},\bm{r})\), \(\alpha\), \(\lambda\),
        \(\widehat{C}_{\mathup{exc}}\), \(k\), and
        \(\theta\in\{1,\frac12\}\);
        \Ensure \((\mathcal{H}+\lambda \mathcal{I})^{-1}\bm{p}\);
 
        \State \(\widehat{\bm q}_1=\operatorname{FFT}(\bm{v})\);
        \State
        \(\widehat{\bm q}_2[\bm\xi]
        =\widehat{\bm q}_1[\bm\xi]/
        (\alpha+\theta k\widehat{C}_{\mathup{exc}}|\bm{\xi}|^2+\lambda)\);
        \State \(\bm{q}_3=\operatorname{iFFT}(\widehat{\bm{q}}_2)\);
        \State \(\bm{q}_4=\bm{r}/\lambda\);
        \State \Return \((\mathcal{H}+\lambda\mathcal{I})^{-1}\bm{p}=(\bm{q}_3,\bm{q}_4)\).
 \end{algorithmic}
\end{algorithm}
For the preconditioner \(\mathcal{M}_\lambda^{(2)}\), it suffices to invert the
independent blocks \(\tilde S_i+\lambda I\). Direct multiplication gives
\begin{equation*}
    \tilde{S}_i^2=-(m_{i;1}^2+m_{i;2}^2+m_{i;3}^2)I
\end{equation*}
and, because \(\tilde S_i^T=-\tilde S_i\),
\begin{equation}\label{eq:local-scaled-orthogonal}
 (\lambda I+\tilde S_i)^T(\lambda I+\tilde S_i)
 =(\lambda^2+\|\m_i\|_2^2)I.
\end{equation}
Thus every shifted local factor is a scaled orthogonal matrix:
\begin{equation}\label{eq:local-conditioning}
 \kappa_2(\lambda I+\tilde S_i)=1,\qquad
 \|(\lambda I+\tilde S_i)^{-1}\|_2
 =(\lambda^2+\|\m_i\|_2^2)^{-1/2}.
\end{equation}
In particular,
\begin{equation*}
    (\lambda I+\tilde{S}_i)(\lambda I-\tilde{S}_i)
    =(\lambda^2+\|\m_i\|_2^2)I
\end{equation*}
and hence
\begin{equation*}
    (\lambda I+\tilde{S}_i)^{-1}
    =\frac{\lambda I-\tilde{S}_i}{\lambda^2+\|\m_i\|_2^2}
\end{equation*}
Let \(p_{i;j}=m_{i;j}/\lambda\), the inverse of each block can be explicitly expressed as
\begin{equation}\label{inverse S}
    (\tilde{S}_i+\lambda I)^{-1}=\frac{1}{\lambda(p_{i;1}^2 + p_{i;2}^2 + p_{i;3}^2 + 1)}\begin{bmatrix}
        1 &  p_{i;3} & -p_{i;2} & -p_{i;1} \\
        -p_{i;3} & 1 & p_{i;1} & -p_{i;2}\\
        p_{i;2} & -p_{i;1} & 1 & -p_{i;3}\\
        p_{i;1} & p_{i;2} & p_{i;3} & 1
    \end{bmatrix}
\end{equation}

The forward action is also matrix free: use the Fourier symbols in
Algorithm~\ref{fh_pre} without inversion and the nodal formulas
\eqref{S}--\eqref{Lambda}. At each Euler step, GMRES solves the saddle system
with the two factors above, followed by
\(\widehat\m_h^{n+1}=\m_h^n+k\bv_h^{n+1}\) and nodal normalization.
The projection enforces unit length and changes the tangent update only by
\(O(k^2)\), so the method is first order but has no exact midpoint energy law.

This distinction motivates a midpoint discretization that preserves the
constraint without projection and admits a discrete dissipation law.

\subsection{Projection-free Crank--Nicolson-type midpoint scheme}\label{sec3.3}
To maintain geometric consistency while eliminating explicit projection, we
use a fully implicit midpoint discretization. This is the
Crank--Nicolson-type branch of the formulation and is called the midpoint
scheme below. The lower-order field is evaluated at the same midpoint.
Freezing a nonlinear lower-order field at \(\m_h^n\), or evaluating a
time-dependent forcing at \(t_n\), would generally reduce the temporal order
to one.

With \(\bv^{n+1/2}=(\m^{n+1}-\m^n)/k\) and
\(\m^{n+1/2}=(\m^{n+1}+\m^n)/2\), the midpoint tangent equation has the
saddle formulation: find
\((\bv^{n+1/2}_h,r_h)\in H_h^1\times Q_h\) such that
\begin{equation}\label{C-N}
    \left\{\begin{aligned}
         &\alpha(\bv_h^{n+\frac{1}{2}},\bphi_h)_h
         +(\m_h^{n+\frac{1}{2}}\times\bv_h^{n+\frac{1}{2}},\bphi_h)_h\\
         &\quad+\widehat{C}_{\mathup{exc}}
         (\nabla\m_h^{n+\frac{1}{2}},\nabla \bphi_h)\\
         &\quad +(\m_h^{n+\frac{1}{2}}\cdot\bphi_h,r_h)_h\\
         &\quad =(1+\alpha^2)
         (\h_{\mathup{low}}(\m_h^{n+\frac{1}{2}},t_{n+\frac{1}{2}}),
         \bphi_h)_h,\\
         &\qquad \forall (\bphi_h,q_h) \in H_h^1\times Q_h,\\
         &(\bv^{n+\frac{1}{2}}_h\cdot\m_h^{n+\frac{1}{2}},q_h)_h=0.
    \end{aligned}
    \right.
\end{equation}
Substituting
\(\m_h^{n+\frac12}=\m_h^n+\frac{k}{2}\bv_h^{n+\frac12}\)
shows that the exchange contribution to the velocity block is
\(\frac{k}{2}\widehat C_{\mathup{exc}}(-\Delta_N)\). The cross product,
tangent constraint, and lower-order field still depend on the unknown
midpoint. At Picard step \(p\), we freeze these quantities at
\(\widehat{\m}_h^{n+1/2,p}\), retain the split exchange term
\(\widehat C_{\rm exc}(\nabla\m_h^n,\nabla\bphi_h)
+(k/2)\widehat C_{\rm exc}(\nabla\bv_h^p,\nabla\bphi_h)\), and solve the
resulting linear saddle system, where
\begin{equation}\label{hatm}
    \widehat{\m}_h^{n+\frac{1}{2},p}\coloneqq\begin{cases}
        \m_h^n+(k/2)\bv_h^{n-\frac{1}{2}},
        &p=0,\\
        \m_h^n+(k/2)\bv_h^{n+\frac{1}{2},p-1},
        &p\ge1,
    \end{cases}.
\end{equation}
For the first time step we set \(\bv_h^{-1/2}=0\); this choice affects only
the nonlinear initial guess, not the converged midpoint solution.

\begin{proposition}[Local convergence of the midpoint Picard iteration]
\label{prop:picard-convergence}
Fix \(N\) and let \(\mathcal V_N(\bw)\) denote the velocity obtained by solving the
linear saddle system with frozen midpoint \(\bw\). Let
\[
 G_{k,n}(\bw)=\m_h^n+\frac{k}{2}\mathcal V_N(\bw).
\]
Suppose a closed ball \(D\subset V_N^3\) satisfies
\(\min_{\bw\in D,j}|\bw(\bx_j)|\ge m_->0\), the lower-order field is
continuously differentiable on \(D\), and \(G_{k,n}(D)\subset D\). Then
\(\mathcal V_N\) is Lipschitz on \(D\): for some finite \(L_{V,N}\),
\[
 \|\mathcal V_N(\bw)-\mathcal V_N(\bz)\|_N
 \le L_{V,N}\|\bw-\bz\|_N,\qquad \bw,\bz\in D.
\]
If \(q_{\rm P}:=kL_{V,N}/2<1\), the Picard iteration
\(\widehat\m^{p+1}=G_{k,n}(\widehat\m^p)\) has a unique fixed point in
\(D\) and
\begin{equation}\label{eq:picard-rate}
\|\widehat\m^p-\m_h^{n+\frac12}\|_N
 \le\frac{q_{\rm P}^p}{1-q_{\rm P}}
 \|\widehat\m^1-\widehat\m^0\|_N .
\end{equation}
\end{proposition}

Indeed, Proposition~\ref{prop:saddle-wellposed}, compactness of \(D\), and
smooth finite-dimensional dependence of the frozen solve give \(L_{V,N}<\infty\);
the stated estimate is the Banach contraction bound. This standard argument
is also used for inexact midpoint solvers in computational micromagnetics
\cite{PRS:CMA:2018}. The constant \(L_{V,N}\) may depend on \(N\), so this is
not a mesh-uniform Picard theorem.
Thus the midpoint scheme uses the same preconditioner as the Euler scheme,
with \(k\widehat C_{\rm exc}/2\) in the exchange block. Starting from
\eqref{hatm}, each Picard step
updates the lower-order field and nodal blocks, solves the preconditioned
saddle system, and sets
\(\widehat\m_h^{p+1}=\m_h^n+(k/2)\bv_h^p\); after convergence,
\(\m_h^{n+1}=\m_h^n+k\bv_h^{n+1/2}\).

\begin{proposition}[Nodal constraint and discrete energy law]
\label{prop:midpoint-energy}
Assume that \(|\m_h^n(\bx_j)|=1\) at every node and that the midpoint
equation \eqref{C-N} is solved exactly. Then
\[
 |\m_h^{n+1}(\bx_j)|=1\qquad\text{for every }\bx_j\in\mathcal G_N.
\]
Suppose additionally that the applied field is time independent and that the
discrete effective field is generated by a quadratic-affine energy \(g_N\),
\[
 Dg_N(\m)[\bphi]=-(\h_{\rm eff,N}(\m),\bphi)_N .
\]
This class includes the discrete exchange, periodic demagnetization, uniaxial
anisotropy, and Zeeman terms used in Section~\ref{sec:energy-constraint}. Then
\begin{equation}\label{eq:midpoint-energy-law}
 g_N(\m_h^{n+1})-g_N(\m_h^n)
 =-\frac{\alpha k}{1+\alpha^2}
 \|\bv_h^{n+\frac12}\|_N^2\le0 .
\end{equation}
\end{proposition}

\begin{proof}
Because \(q_h\) is arbitrary, the nodal constraint in \eqref{C-N} implies
\[
\m_h^{n+\frac12}(\bx_j)\cdot
\bv_h^{n+\frac12}(\bx_j)=0
\]
at every node. Hence
\[
 |\m_h^{n+1}(\bx_j)|^2-|\m_h^n(\bx_j)|^2
 =2k\,\m_h^{n+\frac12}(\bx_j)\cdot
 \bv_h^{n+\frac12}(\bx_j)=0 .
\]
Taking \(\bphi_h=\bv_h^{n+\frac12}\) in the midpoint equation eliminates
the cross product and gives
\[
 \alpha\|\bv_h^{n+\frac12}\|_N^2
 =-(1+\alpha^2)
 Dg_N(\m_h^{n+\frac12})[\bv_h^{n+\frac12}] .
\]
For a quadratic-affine functional, the midpoint identity is exact:
\[
 g_N(\m_h^{n+1})-g_N(\m_h^n)
 =kDg_N(\m_h^{n+\frac12})[\bv_h^{n+\frac12}] .
\]
Combining the last two identities proves \eqref{eq:midpoint-energy-law}.
\end{proof}

\begin{remark}[Effect of an inexact midpoint solve]
\label{rem:inexact-energy}
Suppose an inexact saddle solve has primal and constraint residuals
\(\rho_n\) and \(\chi_n\):
\[
\begin{aligned}
 &\text{left-hand side of the first equation in \eqref{C-N}}
 -\text{right-hand side}
 =\rho_n(\bphi_h),\\
 &(\bv_h^{n+\frac12}\cdot\m_h^{n+\frac12},q_h)_h
 =\chi_n(q_h).
\end{aligned}
\]
Then the length defect satisfies
\[
 \bigl(|\m_h^{n+1}|^2-|\m_h^n|^2,q_h\bigr)_N
 =2k\chi_n(q_h),
\]
and the energy identity becomes
\[
 g_N(\m_h^{n+1})-g_N(\m_h^n)
 =-\frac{\alpha k}{1+\alpha^2}\|\bv_h^{n+\frac12}\|_N^2
 +\frac{k}{1+\alpha^2}
 \left[\rho_n(\bv_h^{n+\frac12})-\chi_n(r_h)\right].
\]
Thus monotonicity is retained whenever
\(\rho_n(\bv_h^{n+\frac12})-\chi_n(r_h)
\le\alpha\|\bv_h^{n+\frac12}\|_N^2\). This condition explains why the
algebraic tolerances and constraint error must be reported together with an
observed energy law.
\end{remark}

The midpoint velocity is tangent to the midpoint configuration, so the method
preserves nodal length without projection. Proposition~\ref{prop:midpoint-energy}
also identifies the assumptions under which energy decay is exact; a
time-dependent applied field or an inexact nonlinear solve introduces the
explicit correction in Remark~\ref{rem:inexact-energy}.

\subsection{Convergence analysis}
\label{subsec:convergence-analysis}

We distinguish three objects that are easily conflated: the full saddle
matrix, the ideal tangent-reduced velocity operator, and the inexact
time-stepping map. Proposition~\ref{prop:saddle-wellposed} treats the first.
We first analyze the second, then use the exact nodal structure of the
constraint block to obtain a separate two-step GMRES bound for the full
preconditioned saddle system.

After scaling out the constant quadrature weight, let
\(Z\in\mathbb R^{3N_g\times2N_g}\) have orthonormal columns spanning
\(\ker B\). Eliminating the multiplier and writing
\(\bv=Z\by\) gives
\[
 A_T\by=Z^T\bm f,\qquad
 A_T=H_T+S_T,\quad H_T=Z^THZ,\quad S_T=Z^TSZ .
\]
Here \(H_T=H_T^T\succ0\) and \(S_T^T=-S_T\).

\begin{proposition}[Classical reduced-space HSS bound]
\label{thm:hss-rate}
For \(\lambda>0\), let
\[
 T_{\lambda,T}=(\lambda I+S_T)^{-1}(\lambda I-H_T)
 (\lambda I+H_T)^{-1}(\lambda I-S_T).
\]
Because \(H_T\succ0\) and \(S_T^T=-S_T\), the classical HSS theorem
\cite{BGN:SJMAA:2003} gives, in the norm
\(\|e\|_{\lambda,S_T}=\|(\lambda I+S_T)e\|_2\),
\begin{equation}\label{eq:hss-rate}
 \|e^{(j)}\|_{\lambda,S_T}
 \le q_T(\lambda)^j\|e^{(0)}\|_{\lambda,S_T},\qquad
 q_T(\lambda)=\max_{\mu\in\sigma(H_T)}
 \left|\frac{\lambda-\mu}{\lambda+\mu}\right|<1.
\end{equation}
Writing
\(\mu_{\min}=\lambda_{\min}(H_T)\),
\(\mu_{\max}=\lambda_{\max}(H_T)\), the parameter minimizing this bound is
\begin{equation}\label{eq:hss-optimal}
 \lambda_*=\sqrt{\mu_{\min}\mu_{\max}},\qquad
 q_*=\frac{\sqrt{\kappa_T}-1}{\sqrt{\kappa_T}+1},\qquad
 \kappa_T=\frac{\mu_{\max}}{\mu_{\min}}.
\end{equation}
For \(M_{\lambda,T}=(\lambda I+H_T)(\lambda I+S_T)\), the defining
factorization also gives
\begin{equation}\label{eq:preconditioned-identity}
 M_{\lambda,T}^{-1}A_T
 =\frac{1}{2\lambda}(I-T_{\lambda,T}),
\qquad
 \sigma(M_{\lambda,T}^{-1}A_T)
 \subset
 D\!\left(\frac{1}{2\lambda},
          \frac{q_T(\lambda)}{2\lambda}\right).
\end{equation}
\end{proposition}

\begin{corollary}[Computable Fourier spectral enclosure]
\label{cor:hss-fourier}
Let
\[
 \beta_N=\alpha+\theta k\widehat C_{\rm exc}
 \max_{\bm\xi\in\mathcal F_N}|\bm\xi|^2,\qquad
 \theta=1\ \text{(Euler)},\quad
 \theta=\tfrac12\ \text{(midpoint)}.
\]
Then \(\sigma(H_T)\subset[\alpha,\beta_N]\) and
\begin{equation}\label{eq:hss-enclosure}
 q_T(\lambda)\le\overline q_N(\lambda):=
 \max\left\{
 \left|\frac{\lambda-\alpha}{\lambda+\alpha}\right|,
 \left|\frac{\lambda-\beta_N}{\lambda+\beta_N}\right|
 \right\}.
\end{equation}
The enclosure is minimized by
\[
 \lambda_{\rm bd}=\sqrt{\alpha\beta_N},\qquad
 \overline q_{N,*}
 =\frac{\sqrt{\beta_N/\alpha}-1}
 {\sqrt{\beta_N/\alpha}+1}.
\]
\end{corollary}
This follows immediately from the Rayleigh quotient and the Fourier symbol
of \(-\Delta_N\); no additional HSS argument is required.

\begin{theorem}[Full saddle two-step contraction and GMRES bound]
\label{thm:full-saddle-gmres}
Assume \(H=H^T\succ0\), let \(\lambda>0\), and suppose that the nodal norms
satisfy
\[
 0<m_-\le\|\m_i\|_2\le m_+,
\]
and define the full HSS error map
\[
 \mathcal T_\lambda
 =(\lambda\mathcal I+\mathcal S)^{-1}
  (\lambda\mathcal I-\mathcal H)
  (\lambda\mathcal I+\mathcal H)^{-1}
  (\lambda\mathcal I-\mathcal S).
\]
Set
\[
\begin{aligned}
 q_H(\lambda)
 &=\max_{\mu\in\sigma(H)}
   \left|\frac{\lambda-\mu}{\lambda+\mu}\right|<1,\\
 \gamma_m(\lambda)
 &=\min_i\frac{2\lambda\|\m_i\|_2}
                   {\lambda^2+\|\m_i\|_2^2},\\
 \vartheta_2(\lambda)
 &=\left[
 1-\frac{(1-q_H(\lambda)^2)\gamma_m(\lambda)^2}
          {\gamma_m(\lambda)^2+1+q_H(\lambda)^2}
 \right]^{1/2}<1 .
\end{aligned}
\]
Then, in the norm
\(\|\bx\|_{\lambda,\mathcal S}
 :=\|(\lambda\mathcal I+\mathcal S)\bx\|_2\),
\begin{equation}\label{eq:full-two-step}
 \|\mathcal T_\lambda^{\,2j}\bx\|_{\lambda,\mathcal S}
 \le\vartheta_2(\lambda)^j
 \|\bx\|_{\lambda,\mathcal S}.
\end{equation}
Consequently, the full stationary HSS iteration converges.

Let
\(\mathcal P_\lambda=\mathcal M_\lambda^{-1}\mathcal A\)
be the actual left-preconditioned saddle operator from
Section~\ref{sec3.2}. For unrestarted GMRES applied to
\(\mathcal P_\lambda\), define its preconditioned residual by
\(\mathfrak r_j=\mathcal M_\lambda^{-1}(\bb-\mathcal A\bx_j)\).
At every even iteration it satisfies
\begin{equation}\label{eq:full-gmres-bound}
 \frac{\|\mathfrak r_{2j}\|_2}{\|\mathfrak r_0\|_2}
 \le \kappa_\lambda\,\vartheta_2(\lambda)^j,\qquad
 \kappa_\lambda=
 \left(\frac{\lambda^2+m_+^2}{\lambda^2+m_-^2}\right)^{1/2}.
\end{equation}
For unit nodal magnetizations,
\(\kappa_\lambda=1\) and
\(\gamma_m=2\lambda/(\lambda^2+1)\).
\end{theorem}

\begin{proof}
Set \(V=\diag(V_H,I)\), where
\(V_H=(\lambda I-H)(\lambda I+H)^{-1}\), and set
\(U=(\lambda\mathcal I-\mathcal S)
(\lambda\mathcal I+\mathcal S)^{-1}\). Then \(U\) is orthogonal,
\(\|V_H\|_2=q_H\), and
\[
 (\lambda\mathcal I+\mathcal S)\mathcal T_\lambda
 (\lambda\mathcal I+\mathcal S)^{-1}=VU .
\]
Let \(\Pi_v\) project onto the velocity variables. For
\(\by=U\bx=(\bp,\bs)\),
\begin{equation}\label{eq:first-loss}
 \|VU\bx\|_2^2
 \le\|\bx\|_2^2-(1-q_H^2)\|\bp\|_2^2.
\end{equation}
The identity \(\widetilde S_i^2=-\|\m_i\|_2^2I\) gives directly
\(\|\Pi_vU(0,\bs)\|_2\ge\gamma_m\|\bs\|_2\).
With \(\bw_0=\bx\), \(\bw_1=VU\bw_0\), and
\(\bz=\Pi_vU\bw_1\), use
\[
 \bz=\Pi_vU(V_H\bp,0)+\Pi_vU(0,\bs).
\]
The reverse triangle inequality gives
\[
 \gamma_m\|\bs\|_2
 \le\|\bz\|_2+q_H\|\bp\|_2 .
\]
Consequently, by Cauchy--Schwarz and
\(\|\bx\|_2^2=\|\bp\|_2^2+\|\bs\|_2^2\),
\[
 \|\bp\|_2^2+\|\bz\|_2^2
 \ge
 \frac{\gamma_m^2}{\gamma_m^2+1+q_H^2}\|\bx\|_2^2 .
\]
Applying \eqref{eq:first-loss} to \(\bw_0\) and then \(\bw_1\) yields
\[
 \|VU\bw_1\|_2^2
 \le \|\bw_0\|_2^2-(1-q_H^2)
      (\|\bp\|_2^2+\|\bz\|_2^2)
 \le \vartheta_2^2\|\bw_0\|_2^2 .
\]
Similarity by \(\lambda\mathcal I+\mathcal S\) and induction prove
\eqref{eq:full-two-step}.

Finally,
\(\mathcal M_\lambda-(\lambda\mathcal I-\mathcal H)
(\lambda\mathcal I-\mathcal S)=2\lambda\mathcal A\) gives
\(\mathcal P_\lambda=(2\lambda)^{-1}
(\mathcal I-\mathcal T_\lambda)\).
The admissible GMRES polynomial \((1-2\lambda z)^{2j}\) yields
\(\|\mathfrak r_{2j}\|_2
\le\|\mathcal T_\lambda^{2j}\mathfrak r_0\|_2\).
The singular values of \(\lambda\mathcal I+\mathcal S\) lie between
\(\sqrt{\lambda^2+m_-^2}\) and
\(\sqrt{\lambda^2+m_+^2}\). Norm equivalence and
\eqref{eq:full-two-step} prove \eqref{eq:full-gmres-bound}.
\end{proof}

\begin{corollary}[Uniform iteration bound for bounded exchange stiffness]
\label{cor:kh-uniform}
Let
\[
 h:=\min_{1\le\ell\le d}\frac{L_\ell}{N_\ell},
\]
and consider any family of grids and time steps satisfying
\begin{equation}\label{eq:parabolic-family}
 \frac{k}{h^2}\le C_0
\end{equation}
for a constant \(C_0\) independent of \(k\) and \(h\). Assume that
\(0<m_-\le\|\m_i\|_2\le m_+\) uniformly over the family, and fix
\(\lambda>0\) independently of \(k\) and \(h\). Define
\[
\begin{aligned}
 \beta_0&=\alpha+\theta\widehat C_{\rm exc}d\pi^2C_0,\\
 q_0&=\max\left\{
 \left|\frac{\lambda-\alpha}{\lambda+\alpha}\right|,
 \left|\frac{\lambda-\beta_0}{\lambda+\beta_0}\right|
 \right\}<1,\\
 \gamma_0&=\min\left\{
 \frac{2\lambda m_-}{\lambda^2+m_-^2},
 \frac{2\lambda m_+}{\lambda^2+m_+^2}
 \right\}>0,\\
 \vartheta_0&=\left[
 1-\frac{(1-q_0^2)\gamma_0^2}
          {\gamma_0^2+1+q_0^2}
 \right]^{1/2}<1,\qquad
 \kappa_0=\left(
 \frac{\lambda^2+m_+^2}{\lambda^2+m_-^2}
 \right)^{1/2}.
\end{aligned}
\]
Then \(q_T(\lambda)\le q_0\), and the unrestarted left-preconditioned
GMRES residual satisfies
\begin{equation}\label{eq:kh-uniform-gmres}
 \frac{\|\mathfrak r_{2j}\|_2}{\|\mathfrak r_0\|_2}
 \le\kappa_0\vartheta_0^j .
\end{equation}
Consequently, for any \(0<\varepsilon<1\), the residual reduction
\(\|\mathfrak r_{2j}\|_2/\|\mathfrak r_0\|_2\le\varepsilon\) is guaranteed
after at most
\begin{equation}\label{eq:kh-uniform-count}
 2\left\lceil
 \frac{\log(\kappa_0/\varepsilon)}
      {\log(1/\vartheta_0)}
 \right\rceil
\end{equation}
GMRES iterations. The bound depends on \(C_0\), the physical coefficients,
the uniform nodal bounds, \(\lambda\), and \(\varepsilon\), but not on the
individual values of \(k\) or \(h\).
\end{corollary}

\begin{proof}
With \(h_\ell=L_\ell/N_\ell\), every Fourier wave number satisfies
\(|\bm\xi|^2\le\sum_\ell\pi^2/h_\ell^2\), while the definition of \(h\)
gives \(\sum_\ell\pi^2/h_\ell^2\le d\pi^2/h^2\).
Thus \eqref{eq:parabolic-family} gives
\(\sigma(H)\subset[\alpha,\beta_0]\), and the Rayleigh--Ritz principle gives
the same enclosure for \(H_T\); hence \(q_T,q_H\le q_0\).
The function \(2\lambda s/(\lambda^2+s^2)\) attains its minimum over
\([m_-,m_+]\) at an endpoint, so \(\gamma_m\ge\gamma_0\).
For \(F(x,y)=(1-x)y/(y+1+x)\), direct differentiation shows that \(F\)
decreases with \(x\) and increases with \(y\). Therefore
\(\vartheta_2=(1-F(q_H^2,\gamma_m^2))^{1/2}\le\vartheta_0\).
Theorem~\ref{thm:full-saddle-gmres} proves
\eqref{eq:kh-uniform-gmres}, and solving
\(\kappa_0\vartheta_0^j\le\varepsilon\) gives
\eqref{eq:kh-uniform-count}.
\end{proof}

\begin{remark}[Bounded exchange stiffness versus an explicit CFL condition]
\label{rem:bounded-stiffness-cfl}
Condition~\eqref{eq:parabolic-family} is a sufficient condition for the
fixed-shift GMRES bound, not a stability restriction on the implicit
discretizations. To distinguish the two, set
\(C_{\rm ex}=\widehat C_{\rm exc}/(1+\alpha^2)\). Linearizing the
exchange-only LLG equation about a constant unit state gives eigenvalues
\(-(\alpha\pm\mathrm{i})C_{\rm ex}|\bm\xi|^2\). For
\(a=kC_{\rm ex}|\bm\xi|^2\), forward Euler has squared amplification
\(1-2\alpha a+(1+\alpha^2)a^2\), and therefore requires
\begin{equation}\label{eq:explicit-exchange-cfl}
 k\widehat C_{\rm exc}\max_{\bm\xi}|\bm\xi|^2\le2\alpha .
\end{equation}
Condition~\eqref{eq:parabolic-family} only bounds the left-hand side by a
finite constant; it does not impose the smallness in
\eqref{eq:explicit-exchange-cfl}. The two conditions share an \(h^2\) scale,
but the corollary gives no uniform conclusion when \(k/h^2\to\infty\). The
full saddle matrix also has a zero multiplier diagonal block and nonzero
constraint couplings, so it is not diagonally dominant.
\end{remark}

\begin{remark}[Failure of uniform contraction along diverging
exchange-stiffness families]
\label{rem:hss-scaling}
Condition~\eqref{eq:parabolic-family} is substantive, rather than a proof
artifact. Consider a constant unit magnetization. Then the tangent basis can
be chosen spatially constant, \(H_T\) and \(S_T\) commute, and the stationary
iteration matrix is normal with
\[
 \rho(T_{\lambda,T})=
 \max_{\mu\in\sigma(H_T)}
 \left|\frac{\lambda-\mu}{\lambda+\mu}\right|.
\]
Both the zero and largest Fourier modes belong to \(\sigma(H_T)\), and
\(\max_{\bm\xi}|\bm\xi|^2\asymp h^{-2}\). With
\(\beta_h=\alpha+\theta k\widehat C_{\rm exc}
\max_{\bm\xi}|\bm\xi|^2\), minimization over \(\lambda>0\) gives
\[
 \min_{\lambda>0}\rho(T_{\lambda,T})
 =\frac{\sqrt{\beta_h/\alpha}-1}
         {\sqrt{\beta_h/\alpha}+1}\longrightarrow1
 \qquad \left(h\to0,\ \frac{k}{h^2}\to\infty\right).
\]
Hence no scalar shift, even if chosen optimally for each \((k,h)\), yields a
contraction factor uniformly below one along such a family. This includes
fixed-\(k\) refinement and the intermediate scaling \(k\asymp h\).
Corollary~\ref{cor:kh-uniform} gives the strongest uniform conclusion
supported by the present splitting: independence of the particular \(k\)
and \(h\) over any bounded-stiffness family. This obstruction does not
exclude faster right-hand-side-dependent GMRES behavior for smooth data;
Section~\ref{sec:random-rhs} examines that distinction directly.
\end{remark}

The actual endpoints of \(H_T\) depend on the nodal tangent basis, so
\(\alpha\) and \(\beta_N\) are bounds rather than asserted eigenvalues.
Moreover, \(\lambda_{\rm bd}\) minimizes only the stationary reduced-space
enclosure. The parameter minimizing GMRES time for the full saddle
preconditioner can differ because it also reflects the multiplier variables,
Krylov acceleration, and implementation cost.

We next quantify temporal and algebraic errors. Let
\(F_N(\m,t)\) be the finite-dimensional tangent vector field obtained after
eliminating the multiplier, and let
\(\Phi^{\rm E}_{k,n}\) and \(\Phi^{\rm M}_{k,n}\) denote the exact one-step
maps of the projected Euler and converged midpoint schemes, respectively.

\begin{lemma}[Local consistency at fixed spatial resolution]
\label{lem:local-consistency}
Assume \(F_N\) is twice continuously differentiable in a neighborhood of a
solution \(\m_N\in C^3([0,T];V_N^3)\), the Euler velocity equation is
locally well posed for sufficiently small \(k\), and the midpoint equation
has a locally unique solution. Then
\begin{align}
 \|\m_N(t_{n+1})-\Phi^{\rm E}_{k,n}(\m_N(t_n))\|_N
 &\le C_{\rm loc}k^2,\label{eq:euler-local}\\
 \|\m_N(t_{n+1})-\Phi^{\rm M}_{k,n}(\m_N(t_n))\|_N
 &\le C_{\rm loc}k^3.\label{eq:midpoint-local}
\end{align}
\end{lemma}
These are the standard local orders of a smooth projected Euler step and the
implicit midpoint rule; see \cite{A:DCDS-S:2008,PRS:CMA:2018,HNW:SODE1:1993}.
We therefore record the assumptions and conclusions but do not repeat the
Taylor-expansion proof.

\begin{theorem}[Inexact-step global error]
\label{thm:time-order}
Choose \(X\in\{\mathrm E,\mathrm M\}\), with \(p=1\) for Euler and \(p=2\)
for midpoint. Assume the local defect is bounded by
\(C_{\rm loc}k^{p+1}\) and the exact one-step map is stable,
\begin{equation}\label{eq:onestep-stability}
 \|\Phi^X_{k,n}(\bu)-\Phi^X_{k,n}(\bv)\|_N
 \le(1+Lk)\|\bu-\bv\|_N .
\end{equation}
Let the computed map \(\widetilde\Phi^X_{k,n}\) satisfy
\begin{equation}\label{eq:step-perturbation}
 \|\widetilde\Phi^X_{k,n}(\bu)-\Phi^X_{k,n}(\bu)\|_N
 \le\eta_n .
\end{equation}
Then, for \(t_n=nk\le T\),
\begin{equation}\label{eq:inexact-global}
\begin{aligned}
\max_{0\le n\le T/k}\|\m_N(t_n)-\widetilde\m_N^n\|_N
&\le e^{LT}\biggl(
\|\m_N(0)-\widetilde\m_N^0\|_N+C_{\rm loc}Tk^p\\
&\hspace{2.4cm}+\sum_{j=0}^{T/k-1}\eta_j
\biggr).
\end{aligned}
\end{equation}
In particular, if \(\eta_n\le\eta\), the algebraic contribution is bounded
by \(e^{LT}T\eta/k\), and the order \(p\) is retained when
\(\eta=O(k^{p+1})\).
\end{theorem}
This is the standard stability--consistency argument: adding and subtracting
the exact one-step map gives
\(e_{n+1}\le(1+Lk)e_n+C_{\rm loc}k^{p+1}+\eta_n\), and discrete Gronwall
gives \eqref{eq:inexact-global}; see
\cite{HNW:SODE1:1993,Thomee2006a}.

For completeness, if a sufficiently regular periodic solution also satisfies
the semidiscrete estimate
\begin{equation}\label{eq:semidiscrete-assumption}
 \sup_{0\le t\le T}\|\m(t)-\m_N(t)\|_{H^r}
 \le C_{\rm sp}N^{\,r-s},\qquad s>r,
\end{equation}
and the preceding temporal bounds hold uniformly in \(H^r\), the triangle
inequality yields the conditional fully discrete estimate
\begin{equation}\label{eq:full-error}
 \max_n\|\m(t_n)-\widetilde\m_N^n\|_{H^r}
 \le C_T\left(
 N^{\,r-s}+k^p+\sum_{j=0}^{T/k-1}\eta_j\right),
 \qquad p=1\ \text{or}\ 2.
\end{equation}
If \(\eta_j\le\eta\), the last term may be replaced by \(T\eta/k\).
For analytic periodic solutions, \eqref{eq:semidiscrete-assumption} may be
replaced by the corresponding exponential Fourier convergence assumption.

\begin{remark}[How solver tolerances enter the estimate]
\label{rem:solver-tolerance}
The quantity \(\eta_n\) is a perturbation of the magnetization after one
complete time step, not a raw GMRES relative residual. A velocity error
\(\tau_{v,n}\) gives \(\eta_n\le Ck\tau_{v,n}\), so
\(\tau_{v,n}=O(k^p)\) preserves order \(p\). Converting a residual tolerance
to \(\tau_{v,n}\) still requires an inverse bound for the saddle operator.
\end{remark}

\begin{remark}[Boundary of the linear convergence result]
\label{rem:full-gmres-boundary}
The Hermitian part of the full saddle matrix is
\(\operatorname{diag}(H,0)\), so the positive-definite HSS theorem cannot be
applied to that matrix in one step. Theorem~\ref{thm:full-saddle-gmres}
instead exploits transfer from the multiplier nullspace to the damped velocity
block over two steps. It is an unrestarted left-preconditioned residual bound,
not a Euclidean field-of-values theorem.
Corollary~\ref{cor:kh-uniform} gives a uniform iteration count for bounded
exchange stiffness \(k/h^2\le C_0\), uniform nonzero nodal bounds, and a fixed
positive \(\lambda\). Remark~\ref{rem:bounded-stiffness-cfl} distinguishes this
solver-complexity condition from an explicit stability threshold. The
random-right-hand-side experiment below shows that mesh-independent iteration counts do not hold
along the tested fixed-\(k\) path, one representative family with
\(k/h^2\to\infty\), although they are observed for the smooth right-hand sides
in the preceding one-step test.
\end{remark}

\section{Numerical experiments}\label{sec4}

\subsection{Numerical setup}\label{sec:numerical-setup}

The Python programs supplied with the manuscript generate all numerical
results and figures. The computations used Python 3.13.5, NumPy 2.1.3,
SciPy 1.15.3, and the serial \texttt{scipy.fft} pocketfft implementation on
an AMD Ryzen 7 9800X3D processor with 32 GB of memory under Windows 11.

Unless stated otherwise, \(\alpha=0.2\), the normalized exchange coefficient
is one, and \(\lambda=0.5\). GMRES is restarted every 50 iterations. Accuracy
tests use relative linear and Picard tolerances of \(10^{-10}\); the
preconditioner comparisons use a relative linear tolerance of \(10^{-8}\).
Each reported CPU time is the median of three solves following one preliminary
solve that initializes the FFT routines. The timings include Krylov
matrix-vector products and preconditioner applications but exclude matrix and
preconditioner setup. The matrix, right-hand side, stopping criterion, grid,
and FFT implementation are identical in all comparisons.

We compare unpreconditioned GMRES with three preconditioners. The first is the
shifted symmetric-part FFT preconditioner
\((\lambda I+\mathcal H)^{-1}\). The second is the Householder tangent-space
preconditioner. The third is the split FFT--HSS preconditioner
\((\lambda I+\mathcal S)^{-1}(\lambda I+\mathcal H)^{-1}\).
For the Householder method, a nodewise orthonormal matrix \(Q\) spans the
discrete tangent space, GMRES solves
\(Q^T(H+S)Q\by=Q^T\bm f\), and the practical preconditioner
\(Q^TH^{-1}Q\) follows \cite[Eq.~(32)]{KrausEtAl2019}. We reconstruct the
multiplier and verify that the relative residual of the full saddle-point
system is below \(10^{-8}\). The stopping criterion for the reduced system is
conservative because \(Q\) is orthonormal and
\(\|Q^T\bm f\|_2\le\|\bm f\|_2\). This comparison is implemented on the
present Fourier grid. Consequently, its timings should not be interpreted as
a reproduction of the finite-element/NGSolve implementation in
\cite{KrausEtAl2019}.

\subsection{Temporal convergence}\label{sec:accuracy}

On \(\Omega=(0,2\pi)^2\), let
\[
 \m^e(\bx,t)=
 \bigl(\sin(t+x)\cos(t+y),\,
       \cos(t+x)\cos(t+y),\,
       \sin(t+y)\bigr).
\]
This field is periodic and satisfies \(|\m^e|=1\). A smooth forcing is obtained
by substituting \(\m^e\), \(\partial_t\m^e\), and
\(\Delta\m^e\) into the tangent equation. We use a
\(64\times64\) grid, integrate to \(T=0.1\), and define the reported error as
the average of the componentwise \(L^\infty\) errors. In the projected Euler
run the manufactured forcing is sampled at the right endpoint \(t_{n+1}\);
this is the endpoint choice covered by Lemma~\ref{lem:local-consistency}.

Figure~\ref{fig:temporal} is consistent with the orders in
Lemma~\ref{lem:local-consistency} and
Theorem~\ref{thm:time-order}; they do not independently verify the theorem's
stability assumptions. From \(k=10^{-2}\) to \(1.25\times10^{-3}\), the
Euler error decreases from \(7.307\times10^{-4}\) to
\(9.141\times10^{-5}\), and the midpoint error from \(7.132\times10^{-6}\)
to \(1.114\times10^{-7}\); observed orders are within \(10^{-3}\) of one
and two. The largest relative
linear residuals over the refinement study are \(8.5\times10^{-11}\) for
Euler and \(5.6\times10^{-11}\) for midpoint, and the maximum nodal length
error of the midpoint solution is \(9.7\times10^{-12}\).

\begin{figure}[t]
\centering
\includegraphics[width=0.62\textwidth]{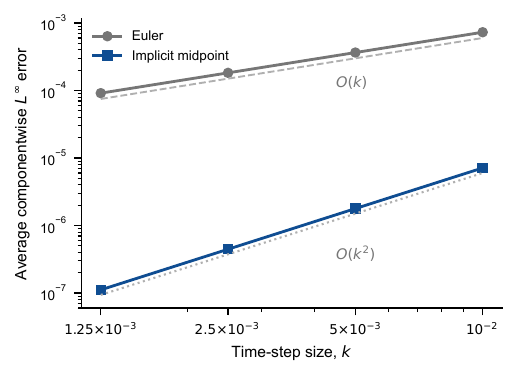}
\caption{Temporal convergence for the periodic manufactured solution.
Dashed and dotted segments show first- and second-order reference slopes.}
\label{fig:temporal}
\end{figure}

\subsection{Preconditioner comparison}\label{sec:preconditioner-results}

We next solve one identical Euler saddle system with \(k=10^{-2}\) and a
smooth manufactured-solution right-hand side on each grid. In two dimensions
\(N=32,64,128,256\); in three dimensions
\(N=16,24,32,48\), with \(N\) points per coordinate. The system has
\(4N^d\) unknowns. Figure~\ref{fig:benchmark} shows the iteration counts and
measured CPU times.

\begin{figure}[t]
\centering
\includegraphics[width=\textwidth]{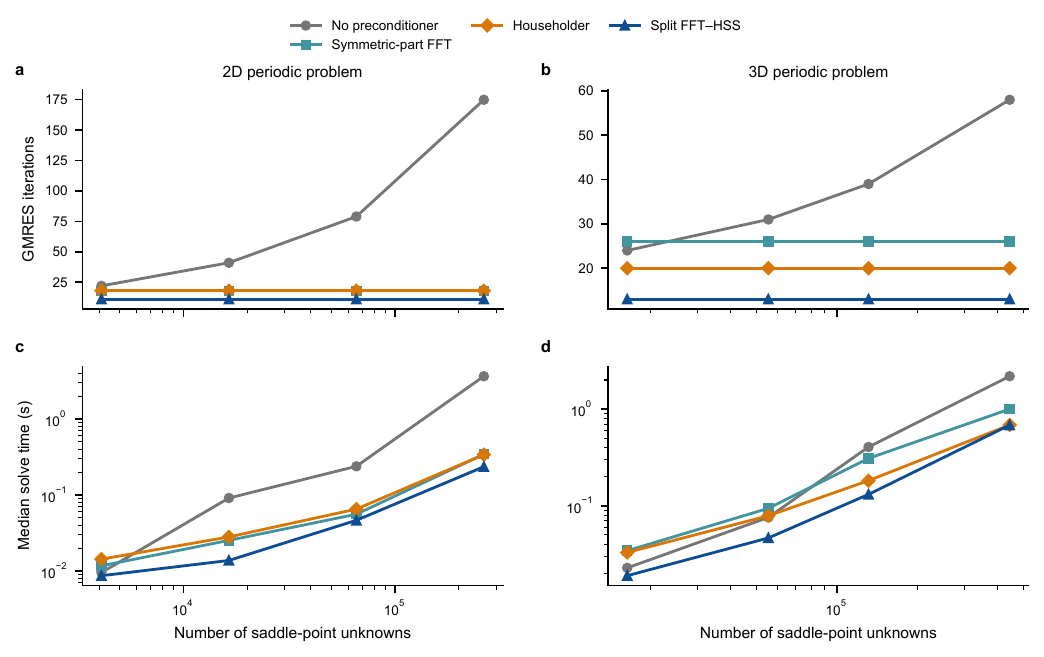}
\caption{Smooth-right-hand-side one-step GMRES comparison. (a,b) Iteration
counts in two and three dimensions. (c,d) Median solve times from three
repeated runs. The split FFT--HSS, Householder tangent-space, and
symmetric-part curves are horizontal for this right-hand side over the tested
meshes; wall-clock timings remain implementation and hardware dependent.}
\label{fig:benchmark}
\end{figure}

For this smooth right-hand side, the unpreconditioned count grows from 22 to
175 in two dimensions and from
24 to 58 in three dimensions. The symmetric-part preconditioner requires
18 and 26 iterations, and the Householder tangent-space preconditioner
requires 18 and 20. Including the local skew/constraint inverse in the full
saddle-point formulation reduces the counts to 11 and 13. On the largest tested
grids, its median times are \(0.236\) s on \(256^2\) and \(0.684\) s on
\(48^3\), giving speedups of 15.45 and 3.20 relative to no preconditioner.
At the smallest
three-dimensional grid, preconditioning overhead can outweigh the iteration
reduction; the timing benefit appears as the problem grows.

\subsection{Dependence on the shift parameter}\label{sec:lambda-results}

For the same smooth right-hand side, Fig.~\ref{fig:lambda} varies \(\lambda\)
from \(2^{-3}\) to \(2^4\) for
\(N=64,128,256\) in two dimensions. The minimum iteration count is 10 at
\(\lambda=0.25\) for all three grids; counts remain between 10 and 12 for
\(0.25\le\lambda\le1\). The iteration counts increase for larger values,
reaching
32, 39, and 42 iterations at \(\lambda=16\). The timing minima are noisier,
but the interval \(0.25\le\lambda\le1\) gives consistently small solution
times. This is a suitable range for the present computations, not
the reduced-space optimum in \eqref{eq:hss-optimal} or the enclosure
minimizer in Corollary~\ref{cor:hss-fourier}. Indeed, the computable enclosure
gives \(\lambda_{\rm bd}=2.07,4.13,8.26\) for
\(N=64,128,256\), respectively, whereas the observed GMRES minimum is
\(0.25\). The difference illustrates the conservatism of the reduced
stationary bound and shows that it does not determine the optimal parameter
for GMRES applied to the full saddle-point system.

\begin{figure}[H]
\centering
\includegraphics[width=\textwidth]{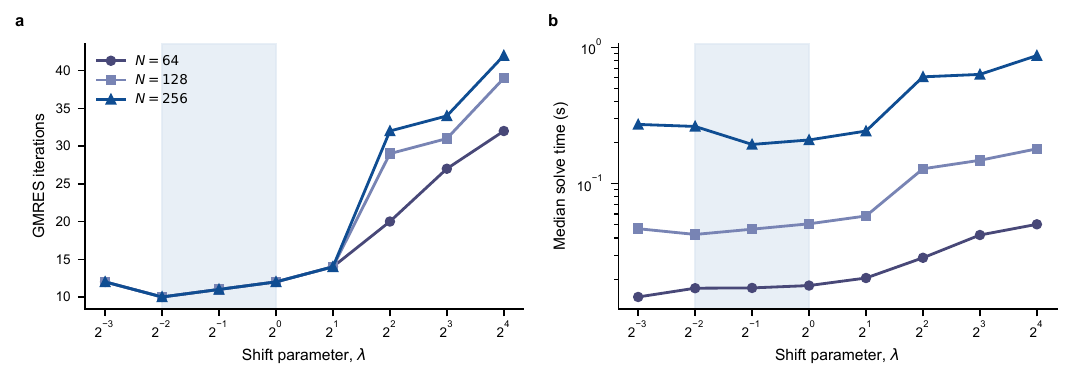}
\caption{Dependence of the split FFT--HSS preconditioner on \(\lambda\).
(a) GMRES iterations. (b) Median solve time from three runs. Shading marks
the interval \(0.25\le\lambda\le1\).}
\label{fig:lambda}
\end{figure}

\subsection{Mesh dependence for random right-hand sides}
\label{sec:random-rhs}

The preceding comparison uses a smooth manufactured right-hand side and
therefore excites only a small set of Fourier modes. To examine the effect of
right-hand-side regularity, we replace the velocity right-hand side by three
deterministic normalized Gaussian vectors, with seeds
20260724--20260726. These vectors contain components throughout the resolved
Fourier spectrum. The multiplier right-hand side is zero, and the relative
residual tolerance for the full system is \(10^{-8}\).
Figure~\ref{fig:broadband} compares the split and Householder methods under
two refinement regimes.

At fixed \(k=10^{-2}\), the median split count increases from 55 at \(N=32\)
to 608 at \(N=256\), whereas the Householder count remains between 86 and 88.
This provides a numerical counterexample to mesh-independent iteration counts
for the split method under fixed-\(k\) refinement. Under the bounded-stiffness
refinement \(kN^2=10.24\), however, the split counts remain 55--56 over the
same grids, while the Householder counts are 86--96. In this second regime the
spectral endpoint \(h_{\max}=5.5248\), the stationary factor
\(q_H=0.8340\), and the two-step full-system factor
\(\vartheta_2=0.9574\) are independent of \(N\). At fixed \(k\),
\(\vartheta_2\) instead approaches one, from 0.9574 to 0.9993.
These results are consistent with the sufficient scaling condition in
Corollary~\ref{cor:kh-uniform}. The conservative factor is not intended to
predict the observed iteration count of restarted GMRES. All reconstructed
relative residuals for the full saddle-point system are below \(10^{-8}\).

The bounded-stiffness experiment remains outside the forward-Euler exchange
range. On \((0,2\pi)^2\), \(\max_{\bm\xi}|\bm\xi|^2=N^2/2\), and hence
\(k\widehat C_{\rm exc}\max|\bm\xi|^2
=(1+\alpha^2)kN^2/2=5.3248\). This is 13.31 times the limit
\(2\alpha=0.4\) in \eqref{eq:explicit-exchange-cfl}. Moreover,
\(\sigma(H)\subset[0.2,5.5248]\), with condition number about 27.6, while the
full saddle matrix retains its zero multiplier block and constraint coupling.
Uniform counts therefore do not follow from explicit stability or diagonal
dominance.

\begin{figure}[t]
\centering
\includegraphics[width=\textwidth]{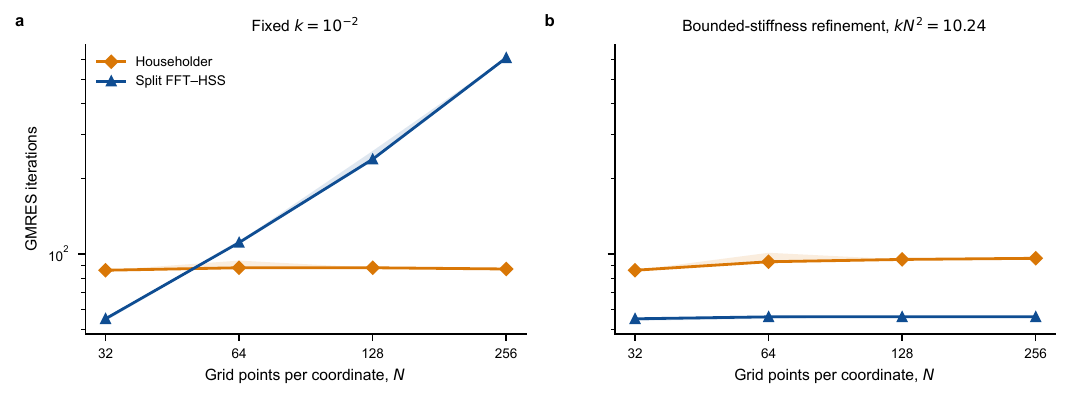}
\caption{Mesh-refinement results for random right-hand sides in two dimensions.
(a) Fixed \(k=10^{-2}\). (b) Bounded-stiffness refinement \(kN^2=10.24\).
Lines show medians over three deterministic Gaussian right-hand sides; shaded
bands show minima and maxima. The vertical axis is logarithmic and all final
relative residuals for the full saddle-point system are below \(10^{-8}\).}
\label{fig:broadband}
\end{figure}

\subsection{Discrete energy decay and length constraint}
\label{sec:energy-constraint}

Finally, we apply the projection-free midpoint scheme to a three-dimensional
periodic micromagnetic problem. The effective field includes exchange,
demagnetization, uniaxial anisotropy, and an applied field. The anisotropy
coefficient is 0.2, the easy axis is \(\bm e_3\), and
\(\h_{\rm ext}=(0.05,0,0)\). We use \(k=0.01\),
\(T=0.1\), \(\lambda=0.5\), and equal grids \(N=16,24,32\). Linear and
Picard tolerances are \(10^{-9}\). The initial condition is the smooth
three-dimensional unit field
\[
 \m^0(x,y,z)=
 \bigl(\sin(x+z)\cos(y+z),\,
       \cos(x+z)\cos(y+z),\,
       \sin(y+z)\bigr).
\]

The normalized energy decreases at every recorded step from 1.8083333333 to
1.7619940172, 1.7619940138, and 1.7619940138 for \(N=16,24,32\);
no monotonicity violation is observed. The maximum length errors are
between \(1.68\times10^{-10}\) and \(1.87\times10^{-10}\), with five Picard
iterations per time step and 7.84--8.88 GMRES iterations per linear solve.
These results assess the energy and constraint conclusions of
Proposition~\ref{prop:midpoint-energy} under inexact algebraic solves. The
largest relative linear residual is \(8.9\times10^{-10}\). To examine the
correction in Remark~\ref{rem:inexact-energy}, we also compute the residual of
the final nonlinear midpoint equation at every time step. Across the three
grids, the maximum nonlinear relative residual is \(9.4\times10^{-10}\).
The largest root-mean-square primal and constraint residuals are
\(3.5\times10^{-9}\) and \(1.5\times10^{-9}\), respectively. The residual
correction reproduces the directly measured energy-balance defect to within
\(5.4\times10^{-16}\); the largest absolute defect itself is
\(3.6\times10^{-11}\). Open boundary conditions and device geometries are not
considered here.

\begin{figure}[H]
\centering
\includegraphics[width=0.90\textwidth]{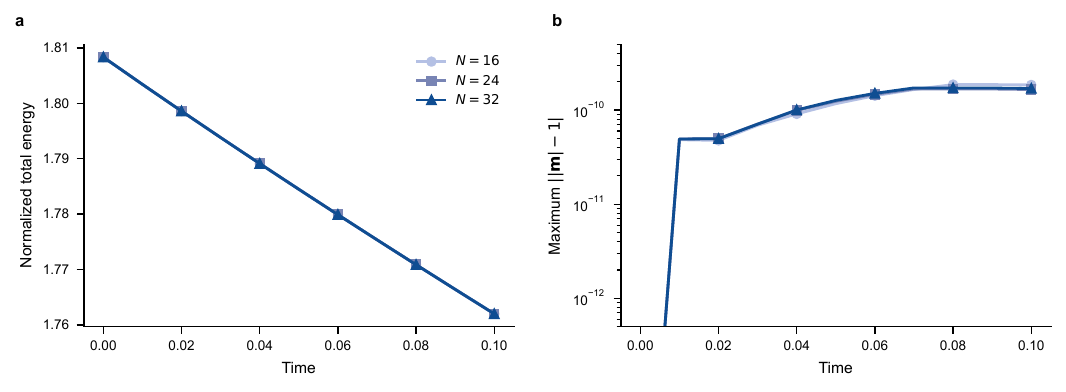}
\caption{Periodic micromagnetic computation with the midpoint scheme:
(a) normalized total energy and (b) maximum nodal unit-length error.}
\label{fig:full-field}
\end{figure}

\section{Conclusions}\label{sec5}

This work developed a matrix-free split FFT--HSS preconditioner for
tangent-plane LLG saddle systems on periodic Cartesian grids. Its main
algorithmic contribution is to combine FFT inversion of the shifted exchange
block with a closed-form \(4\times4\) nodal inverse of the coupled
skew/constraint block. The resulting application requires no global matrix
assembly or local factorization and costs \(O(N_g\log N_g)\) work and
\(O(N_g)\) temporary storage. We proved well-posedness and an even-step GMRES
estimate yielding a uniform iteration bound for bounded-stiffness refinements,
while constant-state analysis identifies the boundary of this uniformity. The
same solver supports projected implicit Euler and projection-free midpoint
stepping; for quadratic-affine energies, the latter preserves nodal length
and satisfies an exact dissipation identity.

Experiments support these claims. On the largest smooth grids, split FFT--HSS
requires \(35\)--\(39\%\) fewer GMRES iterations than same-grid Householder
preconditioning and substantially reduces solve time relative to
unpreconditioned GMRES. Broadband tests retain mesh-independent iterations in
the covered regime and expose deterioration under fixed-step refinement; the
periodic full-field test confirms energy decay and nodal-length preservation.
Together, these results show the benefit of treating the global periodic
operator and local tangent constraint in one preconditioner. Future work will
extend the method to open or irregular geometries and parallel implementations
and develop convergence estimates for restarted GMRES.

\begin{acknowledgements}
Supported by the National Natural Science Foundation of China (12271409),
the Foundation of National Key Laboratory of Computational Physics, and the
Fundamental Research Funds for the Central Universities.
\end{acknowledgements}

\section*{Declarations}

\noindent\textbf{Data availability}\quad
Data will be made available on reasonable request.

\noindent\textbf{Conflict of interest}\quad
The authors declare that they have no conflict of interest.

\noindent\textbf{Generative AI}\quad
\mbox{OpenAI Codex} was used to assist with code refactoring, figure scripting,
and language revision. The authors reviewed and verified all outputs and take
full responsibility for the content of the manuscript.

\bibliographystyle{spmpsci}
\bibliography{references}

\end{document}